\documentclass[12pt,letterpaper,titlepage]{amsart}
\usepackage{amsmath, amssymb, amsthm, amsfonts,amscd,xr}
\thanks{The first author was supported by ERC Advanced Investigators Grant HARG 268105.
The second named author was supported by ISF Grant 1138/10.}
\theoremstyle{plain}
\newtheorem{theorem}{Theorem}[section]

\newtheorem{cor}[theorem]{Corollary}

\newtheorem{prop}[theorem]{Proposition}
\newtheorem{lemma}[theorem]{Lemma}
\newtheorem{definition}[theorem]{Definition}
\theoremstyle{definition}
\newtheorem{ex}[theorem]{Example}
\newtheorem{rmk}[theorem]{Remark}
\numberwithin{equation}{section}
\newtheorem*{theoremA*}{Theorem A}
\newtheorem*{theoremB*}{Theorem B}
\newtheorem*{theoremm1*}{Theorem A'}
\newtheorem*{theoremC*}{Theorem C}
\newtheorem*{theoremD*}{Theorem D}
\newtheorem*{theoremE*}{Theorem E}
\newtheorem*{theoremF*}{Theorem F}
\newtheorem*{theoremE2*}{Theorem E2}
\newtheorem*{theoremE3*}{Theorem E3}

\newcommand{\bs}{\backslash}

\newcommand{\C}{\mathbb{C}}

\newcommand{\Hb}{\mathbb{H}}

\newcommand{\Z}{\mathbb{Z}}

\newcommand{\R}{\mathbb{R}}
\newcommand{\N}{\mathbb{N}}

\newcommand{\Sl}{\operatorname{SL}}
\newcommand{\SO}{\operatorname{SO}}

\newcommand{\GL}{\operatorname{GL}}
\newcommand{\SL}{\operatorname{SL}}
\newcommand{\SP}{\operatorname{Sp}}
\newcommand{\SU}{\operatorname{SU}}

\newcommand{\Lie}{\operatorname{Lie}}
\newcommand{\Ad}{\operatorname{Ad}}

\newcommand{\diag}{\operatorname{diag}}

\newcommand{\re}{\operatorname{Re}}

\usepackage[usenames]{color}

\def\hat{\widehat}
\def\af{\mathfrak{a}}

\def\gf{\mathfrak{g}}
\def\hf{\mathfrak{h}}
\def\kf{\mathfrak{k}}
\def\lf{\mathfrak{l}}
\def\mf{\mathfrak{m}}
\def\nf{\mathfrak{n}}

\def\pf{\mathfrak{p}}
\def\qf{\mathfrak{q}}

\def\sf{\mathfrak{s}}
\def\sl{\mathfrak{sl}}
\def\symp{\mathfrak{sp}}

\def\zf{\mathfrak{z}}
\def\la{\langle}
\def\ra{\rangle}
\def\1{{\bf1}}

\def\U{\mathcal{U}}

\def\oline{\overline}

\def\Unitary{\operatorname{U}}
\def\Field{\mathbb{F}}

\title[Decay of matrix coefficients]
{Decay of matrix coefficients on reductive homogeneous spaces of spherical type}
\subjclass[2000]{22E46, 22F30, 53C35}
\keywords{Lie group, representation, matrix coefficient, homogeneous space, spherical}
\begin{document}
\date{March 10, 2014}

\begin{abstract}
Let $Z$ be a homogeneous space $Z=G/H$ of a
real reductive Lie group $G$ with a reductive subgroup $H$. The investigation concerns
the quantitative decay of matrix coefficients on  
$Z$ under the assumption that $Z$ is of spherical type, that is, minimal parabolic subgroups 
have open orbits on $Z$.  
\end{abstract}

\author{Bernhard Kr\"{o}tz, Eitan Sayag and Henrik Schlichtkrull}
\maketitle
\section{Introduction}

Representation theory provides a concrete way to construct 
functions on a topological group $G$  via 
matrix coefficients. For a continuous linear representation $(\pi, E)$,
say on a Banach space $E$,
and for a vector $v\in E$ and a continuous 
linear functional $\eta\in E^*$ one defines the matrix coefficient 
$$m_{v,\eta}(g):= \eta(\pi(g^{-1})v) \qquad (g\in G)\, .$$
One of the first results on matrix coefficients
is the Gel'fand-Raikov theorem which asserts that the matrix coefficients of irreducible unitary 
representations separate points on $G$. 
\par If $H$ is a closed subgroup of $G$ and $\eta$ is fixed by $H$,
then the matrix coefficient $m_{v,\eta}$ descends to 
a function on the homogeneous space $Z:=G/H$.
Our interest is to obtain sharp
upper bounds for such matrix coefficients on $Z$, under some 
natural assumptions on $G, H$ and the representation $\pi$.

On the geometric side we assume in addition 
that $Z$ is of {\it spherical type}, that is, the set
$PH$ is open in $G$ for some minimal parabolic subgroup $P$ of $G$.  
Standard examples of spaces of spherical type are symmetric spaces, but there 
are others, for instance triple spaces $Z= G/H$ with $H=\SO(1,n)$ diagonally 
embedded into $G=H\times H \times H$.

\par On the analytic side we assume that $E=V^\infty$ is the smooth
globalization of a Harish-Chandra module $V$
(according to Casselman-Wallach).  
Then $E$ is a smooth representation
whose dual $V^{-\infty}:=E^*$
is rather large and can accommodate non-trivial $H$-fixed vectors $\eta$.

\par The first main result, Theorem  \ref{upper bound}, concerns a bound for 
$m_{v,\eta}$ on the subset $PH\subset G$. To be more explicit, let
$K\subset G$ be a maximal compact subgroup, 
$P=MAN$ the Langlands decomposition of the minimal parabolic group $P$, and
$A^+_{P}:=A^{+} \subset A $ the Weyl chamber associated to $P$. 
Then for every $v\in V$
there is a constant $C>0$ such that 
\begin{equation}\label{T1} |m_{v,\eta}(a)| \leq C a^{\Lambda_V} ( 1 + \|\log a\|)^{d_V} \qquad (a\in \oline{A^+}) \, .\end{equation}
Here $\Lambda_V\in \af^*$ (with $\af={\rm Lie}(A)$) 
and $d_V\in \N$ are determined by $V$.

\par To obtain a bound on the whole of $Z$ a further geometric assumption on $Z$ is needed. 
Specifically, let $P_1, \ldots, P_l\supset A$ be the finitely many minimal parabolic subgroups containing $A$ so that $P_j H$ is open, then clearly the geometric condition 
\begin{equation} \label{T2} G= \bigcup_{j=1}^l  K\oline {A^+_{P_j}} H\end{equation}
allows us to deduce from  (\ref{T1}) a global bound on $Z=G/H$
(see Corollary \ref{corsst}). 
Homogeneous spaces for which there is a choice of $A$ such that 
(\ref{T2}) holds true we call {\it strongly 
spherical}. Symmetric spaces are strongly spherical and 
likewise the aforementioned triple spaces.
It is an open problem whether all spherical spaces are strongly spherical. 

\par Let us mention that strongly spherical spaces satisfy the so-called wave-front lemma 
(see Lemma \ref{wfl}), and are thus 
well-suited for lattice counting problems along the lines of \cite{EM}.

\par Strong sphericity combined with a sufficient supply of finite dimensional
$H$-spherical representations (see condition \ref{enough_fd_spherical})  now allows us to obtain a stronger bound in (\ref{T1}),
in which the constant $C$ is uniform with respect to $v$.
Our second main result, Theorem \ref{strong upper bound},
is thus that under these additional assumptions on $Z$,
there exists for each $\eta\in (V^{-\infty})^H$ a
continuous norm $q$ on $V^\infty$ such that
\begin{equation}\label{T3} |m_{v,\eta}(aH)| \leq  q(v)  
a^{\Lambda_V} ( 1 + \|\log a\|)^{d_V} 
\qquad (a\in \oline{A_j^+}) \, \end{equation}
for all $v\in V^\infty$  and $j=1,\dots,l$.
In particular, the assumptions on $Z$ are satisfied in the case of symmetric
spaces (see Theorem \ref{Sstrong upper bound}).
It is interesting to observe that even for the `group case',
where $Z= G \times G /G\simeq G$ is a reductive group,
the bound improves on known bounds
by allowing $v\in V^\infty$ rather than $v\in V$,
see Remark~\ref{group case rmk}.

\par Let us comment about some historical developments and the nature of 
proof of our main results.  In the group case 
the bound  (\ref{T1}) goes back to Harish-Chandra.
There are two different proofs to obtain (\ref{T1}) for $Z=G$:  
one by Casselman-Milicic in \cite{CM}, which uses that the matrix coefficients satisfy 
a regular singular system of differential equations on $A_\C \simeq (\C^*)^n$,
and an approach by Wallach in \cite{W}, just using  ODE-techniques which is of stunning simplicity and brevity. Using the first method van den Ban obtained 
the bound (\ref{T1}) for symmetric spaces.
Our proof of (\ref{T1}) rests on the observation that 
the condition that $PH$ is open allows 
an adaptation of the proof of Wallach. 

\par Combining the assumption of strong sphericity 
and the assumption (\ref{enough_fd_spherical}) about finite dimensional $H$-spherical representations, 
it is possible to construct $K$-invariant 
weight functions $w$ on $Z$ with $w(aH) \asymp  a^{\Lambda_V} (  1+ \|\log a\|)^{d_V} $ 
on $\oline{A^+}.$ Then one can deduce (\ref{T3}) from (\ref{T1}) using the
globalization theorem of Casselman-Wallach. In particular, for symmetric spaces
we thus obtain with (\ref{T3}) an improvement of van den Ban's bounds on generalized matrix coefficients.

Finally, we mention that in \cite{KSS_a} we have studied 
a more qualitative property of decay
of smooth $L^p$-functions which are not necessarily
matrix coefficients of a Harish-Chandra module. 
More precisely, we showed that 
on a reductive homogeneous space
the smooth vectors in the Banach representation 
$L^{p}(Z)$ all belong to the space of continuous 
functions vanishing at infinity. 
The results of the present paper, when restricted to unitary representations,  provide explicit decay results for generalized matrix coefficients.
Therefore, we expect these results to be useful in extending the results of \cite{DRS} beyond symmetric spaces to the realm of 
strongly spherical spaces.

%

\section{Homogeneous spaces of spherical type}

We will denote Lie groups by upper case $LATIN$ letters, e.g. $A$, $B$ etc., 
and their Lie algebras by lower case  ${\frak{German}}$ letters, e.g. $\af$, $\mathfrak b$ etc.  

\par Let $G$ be a real reductive group in the sense of \cite{W}, Sect.~2.1. Further let 
$H<G$ be a subgroup which is reductive in $G$ (as in \cite{KSS_a}). With these data 
we form the homogeneous space of {\it reductive type} $Z:=G/H$.  We denote by $z_0=H$ the 
standard base point of $Z$.

\par We fix a maximal compact subgroup $K<G$  such that $H\cap K$ is a maximal 
compact subgroup of $H$ and such that the associated 
Cartan involution $\theta$ of $G$ preserves $H$.
We will frequently use the Cartan decomposition $\gf=\kf +\sf$ of the Lie algebra $\gf=Lie(G).$ Here $\sf$ is the 
complement of $\kf=Lie(K)$ with respect to a non-degenerate invariant bilinear form 
on $\gf$, say $\kappa(\cdot, \cdot)$. 

\par The form $\kappa$ induces an orthogonal decomposition $\gf=\hf + \qf$ and 
the space $Z$ is topologically a vector bundle over $K/H\cap K.$ Indeed, by
the Mostow decomposition 
\begin{equation}\label{polar}
K \times _{K\cap H} (\sf \cap \qf) \to Z, \ \ [k,X]\mapsto k\exp(X)\cdot z_0\end{equation}
is a diffeomorphism. 
This decomposition is valid
in the generality of reductive homogeneous spaces. 
A smaller class of homogeneous spaces with a richer geometry is introduced below.

\subsection{Spaces of spherical type}

Recall (see \cite{Brion}) that a complex homogeneous space
$G_\C/H_\C$ is said to be spherical if there exists a Borel
subgroup $B_\C$ such that $B_\C H_\C$ is open in $G_\C$. The
following definition is analogous.
Let $Z=G/H$ be a reductive homogeneous space.

\begin{definition} \label{ST}
The space $Z$ is of {\rm spherical type}
if there exists a minimal parabolic subgroup $P$ such that
$PH$ is open in $G$. If in addition $\dim (P\cap H)=0$
then we say that $Z$ is of {\rm pure} spherical type.
\end{definition}

\begin{rmk}  For two closed subgroups $A,B$ of a Lie group $G$, the set
$AB$ is open in $G$ if and only if
$\af + \mathfrak b =\gf$.
Hence the condition that $PH$ is open allows the obvious 
infinitesimal characterization 
$\gf= \hf + \pf$.
\end{rmk}

Note that the main intention behind the
concept in \cite{Brion} is the classification of Gel'fand pairs. With
that intention one should add to Definition \ref{ST}
the condition that $(M,M\cap H)$ is a Gel'fand pair. Here $M=Z_{K}(\af)$ is the centralizer of $\af$ in $K.$ 
However this is not our purpose. The non-symmetric
space $\SP(n,1)/\SP(n)$, for example, is of spherical type
but fails the Gel'fand pair condition.


\begin{definition} \label{SP}
Let $P\subset G$ be a parabolic subgroup.
The pair $(P,H)$ is called {\it spherical} if $PH$ is open in $G$ and for some Langlands decomposition $P=M_PA_PN_P$ 
we have $\mf_P\cap\sf\subset\hf.$
For future reference we write these conditions:
\begin{enumerate}
\item $\mf_P\cap\sf\subset\hf$,
\item $PH$ is open in $G$.
\end{enumerate}
\end{definition}

Equivalently, for some Levi decomposition $P=LN_P$, the largest non-compact semisimple ideal of $L=Lie(L)$ belongs to $\hf$.

\begin{lemma} \label{spherical always minimal}
Let $(P,H)$ be a spherical pair, and let $P_0\subset P$ be minimal
parabolic. Then $P_0H$ is open. 
In particular, $Z$ is spherical.
\end{lemma}

\begin{proof} 
Write $\mf_{P}$ as a direct sum of a compact ideal and non-compact ideal. 
It follows from
condition (1) that  
all non-compact ideals of $\mf_P$ belong to $\hf$.
Now if $P_0=M_0A_0N_0\subset P$ is a minimal parabolic,
then the compact ideals of $\mf_P$ centralize $\af_0$,
hence lie in $\mf_0$. It follows that $M_P\subset M_0H$ and
hence $P_{0}H=PH$. The result follows from condition (2). 
\end{proof}

\subsection{Examples}

\subsubsection{Symmetric spaces} \label{ex:spherical}
In a symmetric space $Z$, all minimal $\sigma\theta$-stable parabolic
subgroups $P$ satisfy (1) and (2), see \cite{BanI}.
Hence by Lemma \ref{spherical always minimal}, $Z$ is of
spherical type and $P_0H$ is open for  
every minimal parabolic $P_0\subset P$.
For more details see remark \ref{rmk-spherical}.

\subsubsection{Gross-Prasad spaces}
\label{STGP}

We let $G_0$ be a reductive group and $H_0<G_0$ a reductive subgroup.
Set $G=G_0 \times H_0$ and $H=\diag (H_0)$. Note that
$Z=G/H\simeq G_0$, viewed as a left$\times$right homogeneous space for $G_0\times H_0$.
\par We consider the following choices for $G_0$ and $H_0$,
with which we refer to $Z$ as a Gross-Prasad space (cf.~\cite{G-P}):

\begin{itemize}
\item $G_0=\GL(n+1,\Field)$ and $H_0=\GL(n,\Field)$
for $n\geq 0$.
\item $G_0=\Unitary(p,q+1,\Field)$ and $H_0=\Unitary(p,q,\Field)$ for
 $p+q\geq2$.
\end{itemize}
Here $\Field=\R$ or $\C$. 

For a parabolic subgroup $P=P_1 \times P_2$ of $G$ the condition that $PH$ is open 
is equivalent to $P_1 P_2$ is open in $G_0$, or $\pf_1 +\pf_2 =\gf_0$. 
The simple verification that this is possible for the above spaces is omitted. 
The spaces in the first item
are pure, but in the second item not in general.

\subsubsection{Triple spaces}\label{Triple}
Let $G_0$ be a reductive group and let $G=G_0^3:= G_0 \times G_0\times G_0$ and 
$H=\diag (G_0)$.  The corresponding reductive homogeneous space 
$Z=G/H$ will be referred to as a triple space.  In general this space is not 
spherical as an easy dimension count will show. For $G_0$ locally $\SO(1,n)$ however, 
one obtains a spherical space. Dimension count shows that it is pure
if and only if $n=2$ or $3$.

\subsubsection{Complex spherical spaces}\label{Brion-ex2}
Let $G_\C/H_\C$ be a complex spherical space with open Borel orbit
$B_\C H_\C$. When we regard the complex groups as real Lie groups,
$G_\C/H_\C$ is of spherical type and $(B_\C,H_\C)$ is a spherical pair.
The complex spherical spaces have been classified
(see the lists in \cite{Kraemer} and \cite{Brion}).
For example, the triple space of $G_0=\Sl(2,\C)$
is a complex spherical space (\cite{Kraemer} p.~152).

\subsubsection{Real forms of spherical spaces}\label{Brion-ex}
Let $G_\C$, $H_\C$, $B_\C$ be as above, and assume that $G$
is a quasisplit real form of $G_\C$. Then $B_\C$ is the complexification
of a minimal parabolic $P$ in $G$, for which $PH$ is open. Hence
$G/H$ is  of spherical type. The triple space
with $G_0=\Sl(2,\R)$ is obtained in this fashion.

\subsubsection{}\label{quasisplit forms}
Let $G_\C/H_\C=\Sl(2n+1,\C)/\SP(n,\C)$ or $\SO(2n+1,\C)/\GL(n,\C)$.
According to \cite{Kraemer} p.~143, these are
complex spherical spaces as in \ref{Brion-ex2}.
Dimension count shows they are pure.
The corresponding split or quasisplit real forms in 
\ref{Brion-ex} are
\begin{eqnarray*}
&&\Sl(2n+1,\R)/\SP(n,\R)\\
&&\SU(n,n+1)/\SP(k,k), \quad n=2k\\
&&\SO(n,n+1)/\Unitary(k,k), \quad n=2k\\
&&\SU(n,n+1)/\SP(k,k+1), \quad n=2k+1\\
&&\SO(n,n+1)/\Unitary(k,k+1), \quad n=2k+1.
\end{eqnarray*}

\medskip
Notice that for $n>3$ the
triple spaces with $G_0=\SO_e(n,1)$ do not
correspond to any spaces in \ref{Brion-ex2} or \ref{Brion-ex}.

\section{Bounds for generalized  matrix coefficients}
In this section  we prove a fundamental
bound for generalized matrix coefficients for spaces of spherical type.

\par To begin with we need to recall a few notions from basic
representation theory. Let $(\pi, E)$ be a Banach representation of
$G$, and let $E^\infty$ denote its space of smooth vectors. As usual
we topologize $E^{\infty}$ by the family of Sobolev norms $\|.\|_k$ for 
$k=0,1,2,\dots$, given by
$$\|v\|_k= \sum_{m_1+\dots+m_n\le k} \|\pi(X_1^{m_1}\cdots X_n^{m_k})v\|$$
with respect to a fixed basis $X_1,\dots,X_n$ for $\gf$.
Then $$E^\infty=\cap_{k\in\N} E_k$$ is an intersection 
Banach spaces, where $E_k$ is the
completion of $E^\infty$ with respect to $\|.\|_k$. 
Likewise the space $E^{-\infty}$ of distribution vectors (i.e.
continuous linear forms on $E^\infty$)
is the union 
$$E^{-\infty}=\cup_{k\in\N} E_{-k},$$ where
$(E_{-k},\|.\|_{-k})$ is the Banach space dual to $(E_k,\|.\|_k)$.
For each $k\in\N$ we thus have
\begin{equation}\label{Sobolev 1}
|\eta(v)|\le \|\eta\|_{-k}\|v\|_k,\quad (\eta\in E^{-\infty}, v\in E^\infty),
\end{equation}
with $\|\eta\|_ {-k}<\infty$ if and only if $\eta\in E_{-k}$.

By continuity of the $G$-action, we conclude
that for each $k\in\N$ there exists $C_k>0$ and $r>0$
such that
\begin{equation}\label{Sobolev}
|\eta (\pi(g)v)|\leq C_k\|\eta\|_{-k} \|v\|_k \|g\|^r ,\quad
(g\in G),
\end{equation}
for all $v\in E^\infty$, $\eta\in E^{-\infty}$. 
Here $\|\cdot\|$ is a norm on $G$ in the sense of
\cite{W} Section~2.A.2, from which Lemma 2.A.2.2 is used.
We recall also that if $\af\subset\sf$ is
a maximal abelian subspace and $\af^+\subset\af$ a Weyl chamber,
then by \cite{W}, Lemma 2.A.2.3, there exist 
$\delta\in\af^*$ and $C>0$ 
such that
\begin{equation}\label{G-norm}
 \|a\| \leq C a^{\delta}
\end{equation}
for all $a\in \oline{A^+}$, the closure of $A^+=\exp(\af^+)$.

\par If $V$ is a Harish-Chandra module for $(\gf, K)$, then we call 
a Banach representation $(\pi, E)$ of $G$ a {\it Banach globalization}
of $V$ provided that the $K$-finite vectors of $E$ are isomorphic
to $V$ as $(\gf, K)$-modules.  The Casselman-Wallach theorem asserts
that $E^\infty$ does not depend on the particular globalization $(\pi, E)$
of $V$ and thus we may define 
$$V^\infty:=E^\infty \text{ and } V^{-\infty}:=E^{-\infty}.$$

\par  Let $Z=G/H$ be a reductive homogeneous space
and $V$ a Harish-Chandra module.
For $v\in V^\infty$ and $\eta\in (V^{-\infty})^H$ an $H$-fixed distribution
vector we denote by
\begin{equation}\label{matrix coeff}
m_{v,\eta}(gH):=\eta(\pi(g)^{-1}v),\quad (g\in G),
\end{equation}
the corresponding
generalized matrix coefficient.
It is a smooth function on $Z$.

Let $P \subset G$ be a minimal parabolic subgroup of $G.$ 
In this situation, there exists a maximal abelian subalgebra $\af \subset \sf$ which is contained in $\pf=\Lie(P).$

We also use the following notations:
\begin{itemize}
\item $\Pi\subset \af^*$ the set of simple roots.
\item  $\pf$ the Lie algebra of $P$ 
\item  $\Sigma^+\subset \af^*$ for the positive system attached to $P.$
\item $P=MAN$ is a Langlands decomposition of $P.$
\item $A^+_{P}:=A^{+} \subset A$ the Weyl chamber associated to $P.$
\item $\bar P=\theta(P)$,  an opposite parabolic subgroup. 
\end{itemize}

We will use  Iwasawa decomposition in the form $G=KAN.$


\par More generally, for a subset $F\subset \Pi$ one defines a standard parabolic 
subalgebra $\pf_F\supset \pf$. Let $\af_F:=\{ X\in \af \mid (\forall \alpha \in F)\alpha(X)=0\}$  
and $\mf_{F}:= \zf_\gf(\af_F)$ be the centralizer of $\af_F$ in $\gf$.  Further  let 
$\Sigma^+\bs \la F\ra$ be the 
set of positive roots which do not vanish on $\af_F$ and let  
$\nf_F:= \bigoplus_{\alpha \in \Sigma^+\bs \la F\ra} \gf^\alpha$
be the corresponding subalgebra of $\nf$. Then $\pf_F:= \mf_{F} \ltimes \nf_F$ defines a 
parabolic subalgebra of $\gf$ containing 
$\pf$.  In particular $\pf_\emptyset=\pf$. 

\par Given a Harish-Chandra module $V$, the quotient $V/\nf_F V $ is non-zero and a Harish-Chandra 
module for the pair $(\mf_F, K_F)$ with $K_F=Z_K(\af_F)$, see \cite{W}, Lemma 4.3.1.
As Harish-Chandra modules are finite under the center of the enveloping algebra, we obtain for every 
$F\subset \Pi$  a finite subset $E(F, V)\subset \af_F^*$ and an integer $d_F\in \N_0$ such that 
$$V/\nf_F V =\bigoplus_{\lambda \in E(F, V)} (V/\nf_F)_\lambda$$ 
with the generalized eigenspaces: 
$$(V/\nf_F V)_\lambda=\{ v \in V/\nf_F V\mid  \forall H\in \af_F: 
( H -\lambda(H))^{d_F} v =0\}\, .$$
We set $E(V):= E(\emptyset, V)$ and record (\cite{W} 4.3.4(2))
$$ E(V)|_{\af_F}= E(F, V)\, . $$

Now label $\Pi=\{\alpha_1, \ldots, \alpha_r\}$ and define $\{ H_1, \ldots, H_r\} $ the corresponding 
dual basis in $\af$. 
We follow \cite{W}, 4.3.5, and define $\Lambda_V\in \af^*$ 
by 
\begin{equation}\label{defi Lambda}
\Lambda_V(H_j):= \max \{-\re \lambda(H_j)\mid \lambda\in E(V)\}\, .
\end{equation}
Furthermore, we let $F_j:= \Pi\bs\{\alpha_j\}$ for $j=1,\dots r$ and put
$$d_V:= \sum_{j=1}^r d_{F_j}.$$

\begin{rmk} The definition of $\Lambda_V$ can be motivated as follows.
By \cite{CM} Thm.~8.22 the leading (with respect to ordering by positive roots)
exponents for $V$ belong to $-E(V)$, and hence
by \cite{CM} Thm.~8.11 every $K$-finite matrix coefficient of $V$ is
bounded on $A^+$ by a constant multiple of
$a^{\Lambda_V}(1+\|\log a\|)^d$ for some $d\in \N$.
Furthermore, \cite{milicic} Thm.~2.1
ensures that $\Lambda_V$ is sharp for the $K$-finite matrix
coefficients. However, these results from \cite{CM} and  \cite{milicic}
are not used in what follows.
\end{rmk}

\def\anumber{s}
\begin{theorem}\label{upper bound}
Let $G$ be a real reductive group and let $H \subset G$ be reductive in $G.$ Suppose that $PH$ is open in $G$ for a minimal parabolic subgroup $P \subset G.$ 
Let $V$ be a Harish-Chandra module.
Then for each $v\in V$ and each $\anumber\in\N$
there exists a constant $C>0$ such that
\begin{equation} \label{fundbound} |m_{v,\eta}(ka)| \le C \|\eta\|_{-\anumber} \,a^{\Lambda_V}
(1+\|\log a\|)^{d_V} \end{equation}
for all $k\in K$, $a \in \oline{A^+_{P}}$ and
$\eta\in (V^{-\infty})^H\cap E_ {-\anumber}$. 
\end{theorem}

\begin{rmk} Note that (\ref{fundbound}) is uniform with respect to
$\eta$ but not with respect to $v$. Under some additional hypotheses
we give in Section~\ref{stronglyspherical} an 
improvement providing uniformity also with 
respect to $v$. 
\end{rmk}

\begin{rmk} The proof will be an adaptation of the proof of
Thm.~4.3.5 in \cite{W} to the present situation   
of generalized matrix coefficients. Note that we have
$m_{v,\eta}(a)=\eta(\pi(a)^{-1}v)$ with a $K$-finite vector $v$,
whereas \cite{W} considers $\mu(\pi(a)v)$ with $\mu$ being $K$-finite.
The main difference is then that \cite{W} has $v\in V^\infty$ arbitrary,
whereas we have $\eta\in V^{-\infty}$ but $H$-fixed.
\end{rmk}

\begin{proof}
Since $H$ is invariant under the Cartan involution,
the assumed openness of $PH$
is equivalent with the same property for $\bar PH$. 
We shall use the property in this form.
The number $\anumber\in\N$ will be fixed
throughout the proof.

We first provide the central step where the proof differs from 
\cite{W}.
Our starting point is the following estimate, which
follows from (\ref{Sobolev}) and (\ref{G-norm}).
Let $(\pi, E)$ be a Banach globalization of $V$.
Then by (\ref{Sobolev}) 
there exists $\delta\in \af^*$ and $C>0$ such that
\begin{equation}\label{1000}
 |m_{v, \eta} (a)|\leq  C \|\eta\|_{-\anumber}\|v\|_{\anumber}\, a^\delta, \qquad
(a\in\oline{A^+})
\end{equation}
for $v\in V$ and $\eta\in (V^{-\infty})^H$.
If $\delta$ happens to be $\leq \Lambda_V$ on $\af^+$ we are
done. Otherwise we need to improve the estimate. The proof will
be by downwards induction along $\af^+$.

The key ingredient is as follows. Suppose that $v\in V$
is of the form
\begin{equation}\label{form of v}
v=d\pi(X) u
\end{equation} 
for some normalized positive root vector
$X\in \gf^\alpha\subset \nf$ and some $u\in V$
(this corresponds to the assumption $\mu\in \nf_F V^\sim$ in \cite{W}
p.116, where $ V^\sim$ is the contragredient of $V$).
As $\bar PH$ is open in
$G$ we can write $X=X_1 + X_2$ with $X_1\in \hf$ and
$X_2 \in \af +\mf +\oline{\nf}$. Now observe that
\begin{eqnarray*} m_{v, \eta} (a) &=& \eta (\pi(a)^{-1} d\pi(X) u) = a^{-\alpha}
\eta(d\pi(X) \pi(a)^{-1} u)\\
&=& a^{-\alpha}\eta(d\pi(X_2) \pi(a)^{-1}u)=a^{-\alpha}\eta(\pi(a)^{-1}
d\pi(\Ad(a) X_2)u)\, .
\end{eqnarray*}
As $\Ad(a)$ is contractive on $\af+\mf+\oline\nf$ we
can write $d\pi(\Ad(a) X_2)u$ as a linear combination
of elements from $V$ with $a$-dependent coefficients,
which are bounded. For vectors of the form
(\ref{form of v}) we thus
obtain with
(\ref{1000}) an improved bound
\begin{equation*}
|m_{v, \eta} (a)|\leq  C' \|\eta\|_{-\anumber}\, a^{\delta-\alpha}  \qquad ( a \in
\oline{A^+})
\end{equation*}
with a constant $C'$ depending on $u$.

\par  Having established the key step, the proof will follow as in \cite{W}, p.~117-118. 
For the sake of completeness we provide the adaptation to the 
present situation. The argument is based on the following simple lemma.

\begin{lemma}\label{diff eq lemma} 
Let $A$ be a complex $n\times n$-matrix. 
There exists $C>0$ such that the following holds.
Let $f:\R\to\C^n$ be differentiable and define $g:\R\to\C^n$ by
\begin{equation}\label{diffeq}
\frac{df}{dt}-Af=g
\end{equation}
Assume
$$\|f(0)\|\le 1,
\quad
\|g(t)\|\leq e^{\nu t}
\quad (t\ge 0),$$ 
for some $\nu\in\R$. Let 
$$\mu=\max\{\re \lambda\mid \lambda \text{ an eigenvalue of } A\}$$
and let $p\in\N$ be the maximal algebraic multiplicity of the eigenvalues with
$\re\lambda=\mu$. Then
$$\|f(t)\|\le C(1+t)^p e^{\max\{\mu,\nu\}t}$$
for all $t\ge 0$.
\end{lemma}

\begin{proof} 
This is easily obtained from the solution formula
for (\ref{diffeq}) and the Jordan form of $A$.
\end{proof}

For the proof of Theorem \ref{upper bound} 
let us assume that for some $\delta\in\af^*$ and $d\in\N$ we have
established for all $v\in V$ a bound
\begin{equation}\label{initial bound}
|m_{v, \eta} (a)|\leq  C \|\eta\|_{-\anumber}\, a^{\delta}(1+\|\log a\|)^d   \qquad ( a \in
\oline{A^+}).
\end{equation}
For elements of the form (\ref{form of v}) we can then improve to
\begin{equation}\label{improved bound}
|m_{v, \eta} (a)|\leq  C' \|\eta\|_{-\anumber}\, a^{\delta-\alpha}  (1+\|\log a\|)^d
\qquad ( a \in
\oline{A^+})
\end{equation}
by the key argument above. The constants $C$ and $C'$ are allowed to
depend on $v$.

Let us write $\delta=\sum_{j=1} ^r \delta_j \alpha_j $ with $\delta_j 
=\delta(H_j)\in \R$. We fix an index $1\leq j\leq r$. 
Note that $$\oline{\af^+}=\bigoplus_{k=1}^r \R_{\geq 0} H_k.$$
We consider $F=F_j=\Pi\bs\{\alpha_j\}$ and 
note that $\af_F = \R H_j $. We
decompose elements $a\in \oline A^+$ as $a= a' a_t$ with $a'^{\alpha_j}=1$ and 
$a_t =\exp(tH_j)$.  

\par Let now $v\in V$ and $\oline v \in V/ \nf_F V$ its coset. 
Let $\oline v_1, \ldots , \oline v_p$ be a basis of the finite dimensional space $\U(\af_F)\oline v$
with $\oline v = \oline v_1$. Define a $p\times p$-matrix 
$B=(b_{kl})$ by $ H_j\oline v_k = \sum b_{kl} \oline v_l$.
Let $v_k\in V$ be representatives of $\oline v_k$ and define 
$w_k:= H_j v_k - \sum b_{kl} v_l \in \nf_F V$. 

Fix $a'\in \oline {A^+}$ and consider the $\C^p$-valued functions
$$F(t):=  
\begin{pmatrix}
m_{v_1, \eta} (a' a_t)\\ \vdots \\ m_{v_p, \eta} (a' a_t)
\end{pmatrix},
\quad
G(t):=  
\begin{pmatrix}
m_{w_1, \eta} (a' a_t)\\ \vdots \\ m_{w_p, \eta} (a' a_t)
\end{pmatrix}.
$$
Then 
$${d\over dt} F(t) = - B F(t) - G(t)$$
and we can apply Lemma \ref{diff eq lemma}. By (\ref{defi Lambda})
the eigenvalues of $-B$ have real part $\leq \Lambda_j$
and multiplicity $\le d_F$. Furthermore,  
by our a priori bound  (\ref{initial bound}) we have 
$$\| F(0) \| \leq C_1 \|\eta\|_{-\anumber}\, (a')^\delta (1+\|\log a'\|)^d $$
for a constant $C_1>0$, independent of $a'$. 
To estimate $\| G(t)\|$ we note that $w_k$ is of the form (\ref{form of v})
so that the improved estimate 
(\ref{improved bound}) can be applied.
Since any root of $\nf_F$ restricted to 
$H_j$ coincides with a
positive integer multiple of $\alpha_j$ this yields that 
$$\|G (t) \| \leq C_2 \|\eta\|_{-\anumber}\, (a')^\delta(1+\|\log a'\|)^d  e^{t (\delta_j -1)}$$ 
for a constant $C_2>0$, also independent of $a'$. 
Combining matters we arrive at 
\begin{equation*}
\|F(t)\| \leq C_3 \|\eta\|_{-\anumber}\, (a')^\delta (1+\|\log a'\|)^d 
(1+t)^{d_F} e^{t \max\{\Lambda_j,\delta_j -1\}}.
\end{equation*}
In particular, 
we conclude that for every $v\in V$ there exists $C>0$ such that
\begin{equation}
|m_{v, \eta} (a' a_t)| \leq C \|\eta\|_{-\anumber}\, (a')^\delta (1+\|\log a'\|)^d 
(1+t)^{d_F} e^{t \max\{\Lambda_j,\delta_j -1\}}
\end{equation}
for all $t\ge 0$ and all $a'$.
We consider now two cases:

If $\delta_j- 1\leq \Lambda_j$ it follows that
we can replace $\delta_j$ by $\Lambda_j$ 
and $d$ by $d+d_F$ in our
initial bound (\ref{initial bound}).

If $\delta_j -1 > \Lambda_j $ we remove the logarithms and find
\begin{equation}
|m_{v, \eta} (a' a_t)| \leq C \|\eta\|_{-\anumber}\,  (a')^\delta (1+\|\log a\|)^d e^{t (\delta_j -\frac12)}.
\end{equation}
Thus we may replace 
$\delta$ by $\delta -\frac12\alpha_j$ in our initial estimate.
In a finite number of steps we arrive in the first case.  

After repeating this process for all $j$, the theorem is proved. 
\end{proof}

\begin{rmk}\label{rmk-spherical}
The theorem applies to symmetric spaces. Suppose that
$Z$ is symmetric and let $\af_q\subset \sf\cap \qf$ be a
maximal abelian subspace (it is unique up to conjugation
by $K\cap H$), and let $\af\subset\sf$ be maximal abelian
with $\af_q\subset\af$. Then
$$\af=\af_q\oplus\af_h:=(\af\cap\qf)\oplus(\af\cap\hf).$$
We choose a positive system for the roots of $\af_q$ in
$\gf$, and a compatible ordering for the roots of $\af$
so that $\af_q^+\subset \oline{\af^+}$ for the positive chambers. 
Then $PH$
is open for the corresponding minimal parabolic $P$.
Thus (\ref{fundbound}) applies to all $a\in \oline{A_q^+}$.
In this situation bounds as
(\ref{fundbound}) have previously
been established in \cite{Ban}, \cite{BSas}.
\end{rmk}

\section{Homogeneous spaces of polar type}

In order to obtain global bounds for the matrix coefficients
we need to assume some further properties of $Z=G/H$.
First of all we require that it allows a 
generalized polar decomposition. We recall that for a 
Riemannian symmetric space $Z=G/K$ the $KAK$-decomposition
of $G$ implies that $Z=KA.z_0$.
 
\begin{definition} \label{PT}
Let $Z=G/H$ be a homogeneous space of reductive type.
A {\rm polar decomposition} of $Z$
consists of an abelian subspace $\af\subset\sf$ and
a surjective proper map
\begin{equation}\label{polar map}
K\times A\ni(k,a)\mapsto kaH\in Z,
\end{equation}
where $A=\exp\af$.
We say that $Z$ is of {\rm polar type}
if such a polar decomposition exists.
\end{definition}

Notice that we do not require $\af\subset\qf$
in Definition \ref{PT}. In general this is not
possible with $\af$ abelian. According to (\ref{polar}) 
every element $z\in Z$ allows a decomposition
$z=k\exp(X)H$ with $k\in K$ and $X\in\qf\cap\sf$, but in general 
the orbits of $K\cap H$ on $\qf\cap \sf$ do not allow a parametrization
by an abelian subspace (for instance in Example \ref{triple spaces} 
below).

Neither do we insist that $\af$ is maximal abelian, 
since in general that would conflict with the properness
of (\ref{polar map}).

\subsection{About properness}
By replacing $\af$ with a subspace complementary
to $\af\cap\hf$, we can arrange $\af\cap \hf=\{0\}$
without affecting the surjectivity
of (\ref{polar map}).
It follows from the corollary below that
then the assumption of properness
in Definition \ref{PT} is superfluous.

\begin{lemma} \label{rep w h-fixed}

Let $G/H$ be a reductive homogeneous space.
There exists a finite dimensional
representation $(\pi, V)$ of $G$ and a vector $v_\hf\in V$
such that $\hf=\{X\in\gf \mid d\pi(X)v_\hf=0\}$.
\end{lemma}

\begin{proof} Follows from Sect. 5.6, Th. 3 in \cite{Ak}.
\end{proof}

Let $\af\subset \sf$ be an abelian subspace.

\begin{lemma} \label{AH closed}
The set $AH$ is closed in $G$. Furthermore,
if $\af\cap \hf=\{0\}$ then
$(a,h)\mapsto ah$ is proper $A\times H\to AH$.
\end{lemma}

\begin{proof}
\par We may assume $\af\cap \hf=\{0\}$ for the entire lemma.
We argue by contradiction. Suppose that $AH$ were not closed
in $G$ or that $(a,h)\mapsto ah$ were not proper.
Then there would exist sequences $(a_n)_{n\in \N}$ in $A$ and
$(h_n)_{n\in N} $ in $H$, both leaving every compact subset,
such that $p=\lim_{n\to \infty} a_n h_n$ exists in $G$.

\par Let $(\pi,V)$ be a finite dimensional
representation as in Lemma \ref{rep w h-fixed}.
Then the limit $\lim_{n\to \infty} \pi (a_n) v_\hf$  exists.
Let $X_{n}=\log(a_{n}).$ Passing to a subsequence we may assume that
$$X:=\lim_{n\to \infty} {X_{n} \over \|X_n\|} \in \af-\{0\}$$
exists and  $\lim_{n \to \infty}  \|X_n\|=\infty.$ We will show that $v_\hf$ is fixed under $d\pi(X)$ contradicting
the assumption that $\af \cap \hf=\{0\}$ and $X \ne 0.$

Indeed, write $v_h$ as a sum of joint eigenvectors for $\af$, say  
$$v_h=\sum_{\mu \in \af^*} v_{\mu}$$

Applying $\pi(a_{n})$ yields $$\pi(a_n)v_{\hf}=\sum_\mu e^{\mu(X_n)} v_{\mu}$$

Since $\pi(a_n)v_{\hf}$ is convergent, it follows that
$\sup_n \mu(X_n) <\infty$ for every $\mu$ for which $v_{\mu}\neq 0.$ With $\lim_{n \to \infty} ||X_n||=\infty$ we get that $$\mu(X)=\lim_{n \to \infty} \frac{\mu(X_{n})}{||X_{n}||} \leq 0$$
for all such $\mu$. 

Applying the same argument to the
convergent sequence $\theta(a_nh_n)$ we find likewise that
$\mu(X)\geq 0$ for all $\mu$ with $v_{\mu} \ne 0.$ 
Thus we obtain that $v_\hf$ is fixed under $d\pi(X)$ which finishes the proof.

\end{proof}

\begin{cor} \label{cor-polar}
The set $KAH$ is closed in $G$, and
if $\af\cap\hf=\{0\}$ then
$(k,a,h)\mapsto kah$ is proper $K\times A\times H\to KAH$.
\end{cor}

\subsection{Examples}\label{examples KAH}

We provide some examples of
homogeneous spaces of polar type.

\subsubsection{Symmetric spaces}\label{SSP}
Symmetric spaces are of polar type. In fact 
let $\af_q\subset \sf\cap \qf$ be maximal abelian,
as in Remark \ref{rmk-spherical}. Then $G=KA_qH$
(see \cite{Sbook} p.~117).

\subsubsection{Triple spaces}\label{triple spaces}
Let $G/H=G_0^3/\diag(G_0)$ with $G_0$ reductive, 
as in Example \ref{Triple}.
Let $K_0<G_0$ be  a maximal compact
subgroup. We fix  an Iwasawa decomposition $G_0=K_0 A_0 N_0$ and
let $P_0= M_0 A_0 N_0$ be the associated minimal parabolic
subgroup. Set $K=K_0\times K_0\times K_0$.

\begin{prop}\label{KAH for triple}
Suppose that $B_0\subset G_0$ is a subset such that
$$G_0=A_0 M_0 B_0 K_0.$$  Then,  for $A=A_0\times A_0\times B_0$,  one
has $G=KAH$.
\end{prop}

\begin{proof} Let $(g_1, g_2, g_3)\in G$. From the
$KAH$-decomposition of the symmetric space $G_0\times G_0/
\diag(G_0)$ we obtain
$$(g_1, g_2)= (g, g) (a_1, a_2) (k_1, k_2)$$
for some $g\in G_0$, $a_1, a_2 \in A_0$ and $k_1, k_2\in K_0$. Now
choose $m\in M_0$, $a_0\in A_0$,  $b_0\in B_0$ and $k_0\in K_0$
such that $g^{-1} g_3= a_0 m_0 b_0 k_0$. Then
$$(g_1, g_2, g_3)=(ga_0m_0, ga_0m_0, ga_0 m_0)(a_0^{-1} a_1,
a_0^{-1} a_2, b_0) (m_0^{-1} k_1, m_0^{-1} k_2, k_0)$$ as
asserted.
\end{proof}

The proposition applies in the following cases:
\begin{itemize}
\item $G_0=\SL(2,\R)$ and 
$$A_0=\big\{\begin{pmatrix} e^t&0\\0&e^{-t}\end{pmatrix}\mid t\in\R\big\},\quad
B_0=\big\{\begin{pmatrix} \cosh s&\sinh s\\\sinh s&\cosh s\end{pmatrix} \mid s\in\R\big\}$$
\item $G_0=\SO_e(1,n)$ and $B_0=\exp(\R X)$ for some $0\neq X\in
 \sf_0\cap \af_0^\perp$.
\end{itemize}
Note that it also applies to $B_0=N_0$ for general $G_0$,

\begin{cor} \label{triple is polar}
Let $G_0=\SL(2,\R)$ or $G_0=\SO_e(1,n)$ for $n\geq 2$ and $Z= G_0^3/ \diag(G_0)$.
Then $Z$ is of polar type.
\end{cor}

\subsubsection{Gross-Prasad spaces}\label{GPspaces}
Let $G/H=G_0\times H_0/\diag(H_0)$ be one of the Gross-Prasad spaces
considered in Example \ref{STGP}.

\begin{lemma} \label{GP is polar}
$G/H$ is of polar type.
\end{lemma}

\begin{proof} We first treat
the case $(G_0, H_0)= (\GL(n+1,\Field),
\GL(n,\Field))$ where $\Field=\R$ or $\C$.
Let us embed $H_0$ in $G_0$ as
the lower right corner.

We define a two-dimensional
non-compact torus of $\GL(2,\Field)$ by
$$B=\big\{\begin{pmatrix} a&b\\b&a\end{pmatrix} \mid a,b\in\R, a>|b| \big\}$$ 


In $\GL(2k,\Field)$ we define a $2k$-dimensional
non-compact torus $A_{2k}$  by $k$
block matrices  of form $B$ along the diagonal.
Explicitly, $$A_{2k}=\{diag(b_{1},...,b_{k}): b_{j} \in B\}$$ 

In $\GL(2k+1,\Field)$ we define $A_{2k+1}$ to consist of similar
blocks together with a positive number in the last diagonal
entry.
Explicitly, $$A_{2k+1}=\{diag(b_{1},...,b_{k},b): b_{j} \in B, b \in \R^{*}\}$$ 
With these choices, when we consider $A_{2k} \subset H_{0}$ as a subgroup of $G_{0}$ using the  lower right corner embedding, we get $A_{2k} \cap A_{2k+1}=\{1\}.$

Finally we let $$A=A_{n+1}\times A_n\subset G=
\GL(n+1,\Field)\times\GL(n,\Field).$$

With $K=\Unitary(n+1, \Field)\times \Unitary(n,\Field)$ we claim
that
$$G=KAH\, , $$
or, equivalently,
$$\GL(n+1, \Field)=\Unitary(n+1,\Field) A_{n+1} A_n \Unitary(n, \Field)\, .$$
We proceed by induction on $n$.
The case $n=0$ is clear.
We shall use the known polar decomposition
for the almost symmetric  pair $(\GL(n+1,\Field), \GL(n, \Field))$:
$$ \GL(n+1,\Field)=\Unitary(n+1,\Field) B_1 \GL(n, \Field)$$
where $B_1$ is the two-dimensional torus of form $B$
located in in the upper left corner.
Now insert for $\GL(n,\Field)$ by induction, but in
opposite order:
$$\GL(n,\Field)=\Unitary(n-1, \Field)  A_{n-1} A_n \Unitary(n,\Field)$$
and observe that $\Unitary(n-1, \Field)$ commutes with $B_1$.

The case with $G_0=\Unitary(p,q+1,\Field)$ where
$\Field=\R$, $\C$ or $\Hb$ is similar. Choose
non-compact Cartan subspaces for $\gf_0$ and $\hf_0$
along antidiagonals, and note that
the overlap between these, as subspaces of $\gf_0$,
is trivial. Now proceed by induction
as before.
\end{proof}

\subsubsection{The spaces
$G/H=\SP(n,\R)/(\SP(n-1,\R)\times \Unitary(1))$}\label{SO23}
\ \par
Consider $G=\SP(n,\R)$ with maximal compact subgroup
$K=\Unitary(n)$. Let
$H\subset L\subset G,$ where
$$L=L_1\times L_2:=\SP(n-1,\R)\times\SP(1,\R),$$
and $H_1=L_1$, $H_2=\Unitary(1)\subset L_2$. The intermediate
quotients $G/L$ and $L/H=L_2/H_2$ are both symmetric.

We use the standard model
$$\symp(n,\R)=\Big\{[X_1,X_2,X_3]:=
\begin{pmatrix} X_1&X_2\\X_3&-X_1^t\end{pmatrix} \Big| 
\begin{matrix}
X_1,X_2,X_3 \in M(n,\R)\, \\
X_2,X_3 \text{ symmetric}
\end{matrix}
\Big\}$$
for $\gf$. 
Then $\lf_1$ consists of 
similar blocks of size one less, embedded as
the upper left corners of $X_1$, $X_2$ and $X_3$ . 
Likewise, $\lf_2=\symp(1,\R)$ consists of 
the matrices from $\sl(2,\R)$
of which the entries embed in the lower right corners
of $X_1$,$X_2$ and $X_3$.

Let $\af\subset\sf$ be the two-dimensional abelian 
space of matrices
$[X_1,0,0]$ with $X_1\in\la E_{11}+E_{nn},E_{n1}+E_{1n}\ra$.
We claim that 
\begin{equation}\label{KAH}
G=KAH
\end{equation}
holds for $A=\exp\af$. Since $[E_{11},0,0]\in\hf$ it is
equivalent that (\ref{KAH}) holds
for the non-abelian space of matrices
$[X_1,0,0]$, with $X_1\in\la E_{nn},E_{n1}+E_{1n}\ra$.
Let 
$$Y_1=[E_{n1}+E_{n1},0,0], \quad Y_2=[E_{nn},0,0]$$
and $A_i=\exp\R Y_i$,
then our claim amounts to $G=KA_1A_2H$.

The intermediate symmetric spaces both allow a polar decomposition.
Specifically, $G=KA_1L$ 
and $L=(L\cap K)A_2H$. 
Hence
\begin{equation*}
G=KA_1(L\cap K) A_2H,
\end{equation*}
and it remains to remove the middle factor.
This is accomplished by showing that it allows
the product decomposition
$$L\cap K=(L\cap K)^{\af_1}(L\cap K\cap H)^{\af_2}.$$
As $L_1\subset H^{\af_2}$, it suffices to 
decompose elements from $L_2\cap K$. This is done on the level of Lie algebras
as follows
$$[0,E_{nn},-E_{nn}]=[0,-E_{11}+E_{nn},E_{11}-E_{nn}]
+[0,E_{11},-E_{11}].$$
An elementary computation shows that the two terms commute with
$Y_1$ and $Y_2$, respectively. Hence (\ref{KAH}) follows and
thus $G/H$ is of polar type with a 2-dimensional $A$.

\begin{rmk} \label{maximal rank}
Note that in \ref{triple spaces} and \ref{GPspaces}
all polar decompositions
$G=KAH$ are with $\af$ a full Cartan subspace in~$\sf$.
In \ref{SO23} this is only the case when $n=2$.
\end{rmk}

\subsection{Relation to spherical spaces}

It seems that there is a close connection between spherical spaces and polar spaces. 
Here we provide an indicator why spherical might imply polar. 

\begin{lemma} \label{interior point}
Let $(P,H)$ be a spherical pair with
Langlands decomposition $P=MAN$ of $P$.
Then there exists $a\in A$ such that
\begin{equation}\label{Sard}
\kf^a+\af+\hf=\gf
\end{equation}
where $\kf^a:=\Ad(a^{-1})(\kf)$.
\end{lemma}

Note that in view of Sard's theorem 
an equivalent formulation of the
conclusion is that
$KAH$ has an interior point.

\begin{proof}
Otherwise $L(a):= \Ad(a)\kf + \af +\hf$ is a proper subspace
of $\gf$ for all $a\in A$. For $t\mapsto a_t$ a ray tending to
infinity  in $A^+$ we note that
$\lim_{t\to \infty} L(a_t)= \mf\cap\kf + \af + \nf+\hf$ in
the Grassmann variety of all subspaces of $\gf$.
By (1) and (2) in Definition \ref{SP} we obtain
$\lim_{t\to \infty} L(a_t)=\gf$. In particular, for large enough $t$ we 
have $L(a_t)=\gf$, a contradiction.
\end{proof}

\section{Strongly spherical spaces}\label{stronglyspherical}

The estimate in Theorem \ref{upper bound}  for matrix coefficients on $Z=G/H$ 
yields bounds for the $m_{v, \eta}$ on subsets of the 
form $K\oline {A^+_P}.z_0\subset Z$  where the associated minimal parabolic 
$P\supset A$ satisfies $PH$ is open. In order to derive a global estimate on $Z$
we therefore need that $Z$ admits not only the polar decomposition
$Z=KA.z_0$ but also the stronger 
\begin{equation}\label{union} 
Z=\bigcup_{P\supset A,\, {P}H\ {\rm open}} 
K\oline{A^+_P}.z_0\,.
\end{equation}
Unfortunately (\ref{union}) is not true for every $A$ which admits $G=KAH$. 

\begin{ex} Let $Z$ be the triple space of $G_0=\Sl(2,\R)$ and $A=A_0\times A_0 \times B_0$,
with $A_0$ the diagonal matrices with positive diagonal entries, 
and with $B_0=\SO_e(1,1)$. We have already seen in Proposition 
\ref{KAH for triple}
that $G= KAH$. It is not hard to see that  
(\ref{union}) fails for this $A$.
Let $P_0$ and $Q_0$ be parabolics containing $A_0$ and
$B_0$, respectively, and let $\oline{P}_0$ be opposite to $P_0$. 
Then $P'=P_0 \times \oline{P}_0 \times Q_0$ is a 
parabolic subgroup which intersects $H$ trivially, hence $P'H$ is open,
whereas $P'':= P_0 \times P_0 \times Q_0$ intersects $H$ in 
positive dimension and $P''H$ is not open. However, 
to attain that $Z=\cup
K\oline{A^+_P}.z_0$ we need the union to encompass parabolics of
both types $P'$ and $P''$.
\end{ex}

To remedy the situation we propose a new notion
which combines polar and spherical types compatibly.

\begin{definition}\label{strongsph}
A reductive homogeneous space $Z=G/H$ is of
{\rm strong spherical type} if the following
holds. There exists a maximal abelian subspace $\af\subset\sf$ 
and minimal parabolic subgroups
$P_1,\dots,P_l\supset A=\exp\af$ such that $P_jH$ is open
for all $1\leq j\leq l$ and such that 
\begin{equation}\label{union j}
Z=\bigcup_{j=1}^l K\oline{A^+_{P_j}}.z_0\,.
\end{equation}
\end{definition}

In particular, $Z$ is then both spherical and polar. The following
is now an immediate consequence of Theorem \ref{upper bound}.

Let $V$ be a Harish-Chandra module and let 
let $\Lambda_{j,V}\in\af^*$ be defined by 
(\ref{defi Lambda}) with respect to $P_j$,
that is, if $P_j=s_jPs_j^{-1}$ for $s_j$ in the
normalizer of $\af$ in $K$, then 
$\Lambda_{j,V}=\Lambda_V\circ\Ad(s_j^{-1})$.

\begin{cor}\label{corsst}
Assume that $Z=G/H$ is of strong spherical type. 
Then for each $v\in V$ and 
each $\anumber\in\N$
there exists a constant $C>0$ such that
\begin{equation}\label{global bound}
|m_{v,\eta}(ka.z_0)|\leq C\|\eta\|_ {-s} \,a^{\Lambda_{j,V}}
(1+\|\log a\|)^{d_V},
\end{equation}
for all $k\in K$, $a\in \oline{A^+_{P_j}}$ 
and
$\eta\in (V^{-\infty})^H\cap E_ {-\anumber}$. 
\end{cor}

\subsection{Examples}
\subsubsection{Symmetric spaces are strongly spherical} \label{SSS}
As in Remark~\ref{rmk-spherical} let $\af_q\subset\sf\cap\qf$ 
and $\af\subset\sf$ be maximal abelian subspaces with $\af_q\subset\af$,
then we have seen in Example \ref{SSP} that $$Z=KA_q.z_0.$$
Furthermore, let $\af^+_{qj}$ for $j=1,\dots,l$ be the Weyl chambers of $\af_q$
corresponding to all (up to $K\cap H$-conjugacy)
the positive systems $\Sigma^+_{qj}$ for the roots of 
$\af_q$ in $\gf$. For each $j$ we choose a compatible positive system
$\Sigma^+_j$ for the roots of $\af$ in $\gf$, and denote by
$P_j$ the corresponding minimal parabolic subgroup of $G$.
Then $P_j$ is contained in the minimal $\sigma\theta$-stable 
parabolic subgroup corresponding to $\Sigma^+_{qj}$, and 
it follows from
Example \ref{ex:spherical} that $P_jH$ is open.
Finally 
\begin{equation}\label{union of chambers}
A_q=\bigcup_{j=1}^l \oline{A^+_{qj}},
\end{equation}
and since
$\oline{A^+_{qj}}\subset \oline{A^+_{P_j}}$, we obtain (\ref{union j}).

\subsubsection{Gross-Prasad spaces}
The simplest of these spaces is 
$$G/H=G_0\times H_0/\diag(H_0)$$
where $G_0=\GL(2,\R)$ and $H_0=\GL(1,\R)$ (see Sections \ref{STGP}, \ref{GPspaces}).
This space is strongly spherical with $A$ chosen as in the proof of 
Lemma \ref{GP is polar}.
We expect other Gross-Prasad spaces are strongly spherical.

\subsubsection{Triple space}
The triple space attached to $G_0=\Sl(2,\R)$  is strongly spherical. In \cite{DKS} 
we show that for any choice of $\af=\af_1 \oplus \af_2 \oplus \af_3$ with $\af_i \subset \sf_0$ 
one-dimensional subspaces and not all equal, one has $G=KAH$. Furthermore,
if all the one-dimensional subspaces are different from each other, then $PH$ is open for
all parabolics $P$ containing $A$. Thus $Z$ is strongly spherical.

\subsection{The wave front lemma}

The following result was proved for symmetric spaces in \cite{EM}, Thm.~3.1,
under the name of "wavefront lemma". It plays a crucial role in that paper.

Let $Z=G/H$ be of strong spherical type, and let 
$P_1,\dots,P_l\supset A$ be as in Definition \ref{strongsph}
so that $G=\cup_{j=1}^l K\oline{A^+_{P_j}}H$.

\begin{lemma}\label{wfl}
For every  neighborhood $V$ of $\1$ in $G$, there exists 
a neighborhood $U$ of $\1$ such that
$$ V g.z_0 \supset gU.z_0$$
for all $g\in \cup_{j=1}^l K\oline{A^+_{P_j}}$.
\end{lemma}

\begin{proof} We may assume that $V$ is $\Ad(K)$-invariant.
By  (\ref{union j}) we reduce to the 
case $g=a\in \oline{A^+_{P_j}}$ for a
minimal parabolic $P_j\supset A$ such that 
$P_jH$ (and hence also $\bar P_j H$) is open.
Let $U_j$ be a neighborhood of ${\bf 1}$ in $\bar P_j$ which 
is contained in $V$ and which is
stable under conjugation 
with elements from $\oline{A^+_{P_j}}$.
As $\bar P_jH$ is open, we see that $U_j.z_0$ is a
neighborhood of $z_0$. Then 
$$Va.z_0\supset U_ja.z_0=aU_j.z_0\supset aU.z_0$$
where $U=\cap_{j=1}^l U_jH$.
\end{proof}

\subsection{Weights on $Z$}
In this subsection we let $Z=G/H$ be a reductive homogeneous space.
In the context of strongly spherical spaces we aim for a more quantitative 
bound in Theorem \ref{upper bound}, in which the constant $C$ 
depends continuously on $v$ in 
the $V^\infty$-topology. For that the concept of weight will be useful.

We fix a norm $\|\cdot\|$ on $G$ (see \cite{W}, Section 2.A.2).  
By a {\it weight} on $Z=G/H$ we shall understand  a locally bounded  
function  $w: Z \to \R_{>0}$ such that there exists constants $C>0$, $N\in \N$ with  
$$ w(gz)\leq C \|g\|^N w(z) \qquad (g\in G, z\in Z)\, .$$
The following is an easy way to construct a weight on $Z$.

\begin{lemma}\label{inf weight}
Let $w(g.z_0):=\inf_{h\in H} \|gh\|$ for $g\in G$. Then 
$w$ is a weight on $Z$. Furthermore, there exist constants
$c_1,c_2,C_1,C_2>0$ such that 
\begin{equation}\label{Mostow norm}
C_1e^{c_1\|X\|}\leq w(k\exp(X).z_0)\leq C_2e^{c_2\|X\|}
\end{equation}
for all $k\in K$ and $X\in\sf\cap\qf$.
\end{lemma}

Note that (\ref{Mostow norm}) applies to 
every element in $Z$ by (\ref{polar}).

\begin{proof} 
We have $w\ge 1$ since $\|g\|\ge 1$ for all $g\in G$.
As $\|xy\|\le \|x\|\|y\|$ for $x,y\in G$
the first statement follows. 

There exist constants $c_1,c_2,C_1,C_2>0$ such that
$$C_1e^{c_1\|Y\|}\leq \|\exp(Y)\| \leq C_2e^{c_2\|Y\|}$$
for all $Y\in\sf$. Hence the
second inequality in (\ref{Mostow norm}) is clear.
For the first inequality we need to show that
$$C_1e^{c_1 \|X\|} \leq \|\exp(X)h\|$$
for all $h\in H$. By Cartan decomposition of $H$ we reduce 
to $h=\exp(T)$ where $T\in\sf\cap\hf$. Let $Y\in\sf$ be determined by
$\exp(X)\exp(T)\in K\exp(Y)$, then $\|X\|\leq \|Y\|$
since the sectional curvatures of $K\backslash G$ are $\leq 0$
(see \cite{Helgason}, p. 73) and $X\perp Y$.
Now 
$$C_1e^{c_1 \|X\|} \leq C_1e^{c_1\|Y\|}\leq \|\exp(Y)\|=\|\exp(X)\exp(T)\|$$
as claimed.
\end{proof}

\par Let $w$ be a weight on $Z$. From the definition we readily obtain that $w^{-1}$ is a weight. More generally 
$w_\alpha(z):= w(z)^\alpha$  defines a weight for all $\alpha \in \R$. 
Further if $w$ and $w'$ are weights then so is $w\cdot w'$. If $w(z)\geq c$ for 
some $c>1$ and all $z\in Z$, then $\log w$ is a weight as 
well.

\par A more refined construction of weights than that of Lemma
\ref{inf weight}
goes as follows. 
Let $U$ be a finite dimensional $G$-module with a non-zero $H$-fixed
vector $u_H\in U$ (see Lemma \ref{rep w h-fixed}).
Such a representation will be referred to as $H$-{\it spherical}.
Set 
\begin{equation}\label{defi w_U}
w_U(g.z_0):= \| g\cdot u_H\|\qquad (g\in G)\, ,
\end{equation}
then $w_U$ is a weight.

\begin{lemma} \label{H-spherical weights}
Let $U$ be $H$-spherical and irreducible.
Let $P=M_PA_PN_P$ be a parabolic subgroup of $G$ for which $PH$ is open. 
Let $\lambda\in\af_P^*$ be the highest $\af_P$-weight 
of $U$ and  $A^+_P\subset A_P$ the positive chamber, both
with respect to $P$.
Then there exist constants $C_1, C_2>0$ such that 
\begin{equation} \label{shw} C_1 a^{\lambda}\leq  w_U(ka.z_0) \leq C_2 a^{\lambda}  
\quad ( a\in \oline{A^+_P}, k\in K)\, .\end{equation}
\end{lemma}

\begin{proof} We may choose an inner product on $U$ such that 
$$\la X.u,v\ra =\la u, -\theta(X).v\ra$$ for $X\in\gf$.
In particular, the norm is then $K$-invariant. 
Let $U_\lambda\subset U$ be the $\lambda$-weight space for $\af_P$, then
$U=\U(\bar n_P)U_\lambda$ and hence all the
$\af_P$-weights $\mu$ in $U$ are obtained from $\lambda$ by
subtracting positive combinations of positive $\af_P$-roots. It follows that
$a^\mu\leq a^\lambda$ for all $a\in \oline{A_P^+}$.
By expanding $u_H$ into $\af_P$-weights
we conclude the second inequality of (\ref{shw}).

Note that 
$u_H$ cannot be orthogonal to $U_\lambda$. Otherwise, as
$U_\lambda$ is $P$-invariant, $\pi(g)u_H$ would be
orthogonal to $U_\lambda$
for all $g$ in the open set $\bar PH$, contradicting irreducibility.
Hence $\la u_H, u\ra \neq 0$ 
for some $u\in U_\lambda$. Now 
$$a^\lambda |\la u_H, u\ra|=|\la u_H, a\cdot u\ra|=|\la a\cdot u_H, u\ra|\leq \|u\| w_U(a).$$
for $a\in A_P$, and the first inequality of (\ref{shw}) follows.
\end{proof}

\subsection{Symmetric spaces}
In this section we assume that $Z=G/H$ is a symmetric space
and use the notation from Example \ref{SSS}. In particular 
\begin{equation}\label{union A_q}
A_q=\cup_{j=1}^l \oline{A_{qj}^+}.
\end{equation}
We fix a chamber $A_q^+$ (for example $A_ {q1}^+$) for
reference, and choose Weyl group elements $s_j$ such that
$A_{qj}^+=\Ad(s_j) A_q^+$ for $j=1,\dots,l$.

\begin{lemma}\label{symm H sph span} 
Assume that $G/H$ is a symmetric space.
For each $\Lambda\in\af_q^*$ and all $d\in\Z$
there exists a 
weight $w$ and a constant $C>0$ such that 
\begin{equation}\label{weight bounds}
a^{s_j\Lambda}(1+\|\log a\|)^d\leq w(ka.z_0) \leq C  a^{s_j\Lambda}(1+\|\log a\|)^d
\end{equation}
for all $k\in K$, $a\in \oline{A^+_{qj}}$ and $j=1,\dots,l$.
\end{lemma}

\begin{proof} It follows from the work of Hoogenboom 
(see \cite{BanII}, Section~5) that 
$\af_q^*$ is spanned by the restrictions of the highest weights of
$H$-spherical representations. 
Hence $\Lambda=c_1\lambda_1+\dots+c_k\lambda_k$
for some $c_1,\dots,c_k\in\R$, where $\lambda_1,\dots,\lambda_k\in \af_q^*$,
are highest weights with respect to $A_q^+$ of 
irreducible $H$-spherical representations $U_1,\dots,U_k$.
The highest weight of $U_i$ 
with respect to $A^+_ {qj}$ is then $s_j\lambda_i$.
It follows from Lemma \ref{H-spherical weights}
that
$$
C_1a^{s_j\lambda_i}\leq w_{U_i}(ka.z_0) \leq C_2a^{s_j\lambda_i}
$$
for $a\in A_{qj}^+$.
With $w$ a multiple of $\Pi_i w_{U_i}^{c_i}$ we obtain (\ref{weight bounds})
for $d=0$.

Select $\lambda_0\in\af_q^*$ such that
$\lambda_0(X)\ge \|X\|$ for all $X$ in the cone $\af_q^+$.
By applying the proved version of (\ref{weight bounds}) 
we see that
there exist a weight $w_0$ and a constant $C_0>0$ such that
$$
a^{s_j\lambda_0}\leq w_0(ka.z_0) \leq C_0 a^{s_j\lambda_0}
$$
for $a\in A_{qj}^+$ and all $j$, and hence
$$
e^{\|\log a\|}\leq w_0(ka.z_0) \leq C_0 e^{\|\lambda_0\|\|\log a\|}
$$
for all $a\in A$. In particular, $w_0\ge 1$, hence
$\log(cw_0)$ is a weight for every $c>1$.
Taking logarithms we thus find a weight $w_1$
and a constant $C'_0>0$
for which
$$
1+\|\log a\|\leq w_1(ka.z_0) \leq C_0'(1+\|\log a\|)
$$
for all $a\in A$. Now (\ref{weight bounds}) follows
by multiplication of the previously found weight with $w_1^d$.
\end{proof}

\par To any weight $w$ we associate the Banach space  
$$E_w:=\{f\in C(Z) \mid \|f\|_{w}:= \sup_{z\in Z} w(z) |f(z)|<\infty\}\, .$$
The group $G$ acts on $E_w$ by left displacements in the arguments, say $\pi(g)f (z):=f(g^{-1}z)$ and we have 
$\|\pi(g)\| \leq C \|g\|^N$. 
Thus the smooth vectors $E_w^\infty$  form  an $SF$-representation of $G$  in the sense of \cite{BK} (that is a
smooth Fr\'echet representation of moderate growth). 

In the following theorem the linear form $\Lambda_V$ is defined 
by (\ref{defi Lambda})
with respect to an open chamber of $A$, which is compatible
with the fixed chamber~$A^+_q$.

\begin{theorem}\label{Sstrong upper bound}
Suppose that $Z=G/H$ is symmetric. 
Let $V$ be a Harish-Chandra module
and fix $\eta\in (V^{-\infty})^H$. 
Then there exists a continuous norm $q$ on $V^\infty$ 
such that 
\begin{equation} \label{fundbound2} |m_{v,\eta}(a)| \leq q(v)  a^{\Lambda_V}
(1+\|\log a\|)^{d_V} \end{equation}
for all $a \in \oline{A_{q}^+}$, and $v\in V^\infty$.
\end{theorem}

Note that for this case it is known that
$\dim (V^{-\infty})^H< \infty$ (see Corollary 2.2 of \cite{Ban}).

\begin{proof} 
We use the parametrization
of the chambers of $A_q$ from (\ref{union A_q}).
According to (\ref{global bound}) we obtain for all $k\in K$,
$a \in \oline{A_{q,j}^+}$ and $j=1,\dots,l$ 
that 
\begin{equation} \label{up} |m_{v,\eta}(ka.z_0)| \leq C_v a^{s_j^{-1}\Lambda_V} 
( 1 + \|\log a\|)^{d_V}\, . \end{equation}
It follows from Lemma \ref{symm H sph span}
that 
there exists a weight $w$ on $Z$ such that
$$ w(ka.z_0) \asymp  a^{-s_{j}^{-1}\Lambda_V}  
(1 + \| \log a\|)^{-d_V} \qquad (k\in K, a \in \oline{A_{q,j}^+}),$$
for all $j$, and hence the product
$wm_{v,\eta}$ is bounded on $Z$ for all $v\in V$. Hence we obtain
an embedding 
$$V\hookrightarrow E_w^\infty, \ \ v\mapsto m_{v,\eta}\, .$$
By the Casselman-Wallach globalization theorem (see \cite{W}, Thm.~11.6.7
or \cite{BK}) 
this embedding extends to a continuous 
embedding of Fr\'echet spaces $V^\infty\hookrightarrow E_w^\infty$.  In particular, there exists 
a continuous norm $q$ on $V^\infty$  such that 
$$ \|m_{v, \eta}\|_{w} \leq q(v) \qquad (v\in V^\infty)\, .$$
Unwinding the definition of the norm in $E_w$ we retrieve
(\ref{up}) with $C_v$ replaced by $q(v)$. Now
(\ref{fundbound2}) follows.  
\end{proof}

\subsection{Group Case}\label{group case} 
We explicate Theorem \ref{Sstrong upper bound} for the 
group case $Z= G \times G / G$. For that let $W$ be a Harish-Chandra module 
for $(\gf, K)$ and $\tilde W$ its contragredient. With that we form the 
Harish-Chandra module $V:= W \otimes \tilde W$ for $(\gf \times \gf , K \times K)$. 
We view $V$ as a submodule of ${\rm End} (W)$ and identify $V^\infty$ as a subspace 
of ${\rm End}(W^\infty)$. To be more precise $V^\infty$ identifies with the 
rapidly decreasing matrices as follows: Choose a Hilbert globalization $E$ of $W$ and with respect to the 
Hilbert structure an orthonormal basis $v_1, v_2, \ldots $ of $E$ consisting of 
vectors $v_i$ which belong to $K$-types $\tau_i \in \hat K$ with $\tau_i \leq \tau_j $ for $i\leq j$.
This identifies $W^\infty$ with the standard nuclear Fr\'echet space 
$$s(\N):=\{ (x_n)_{n\in \N} \in \C^\N \mid \sup_{n\in \N} n^k |x_n|<\infty, \forall k\in \N \}$$
of rapidly decreasing sequences (see \cite{Fokko}, p.~290).
A continuous linear map $T: s(\N) \to s(\N)$ is thus given by a matrix $T=(t_{n,m})_{n,m}$  
and we say that $T$ is rapidly decreasing provided that 
$\|T\|_k:=\sup_{n,m} |t_{n,m}| (n+m)^k  <\infty$ for 
all $k\in \N$.  Now $V^\infty$ is the space of such maps, and its
topology is defined by the norms $\|\,.\,\|_k$
In particular the trace map
$$\eta: V^\infty \to \C, \ \ T \mapsto {\rm tr} (T)=\sum_n t_{nn} $$
is a continuous linear functional on $V^\infty$, which is fixed by the diagonal 
subgroup $H:=\diag(G)< G\times G$. 

\par Let $\af\subset\sf$ be maximal abelian, and put
$$\af_q=\{(X,-X)\in\gf\times\gf\mid X\in\af\}$$
then $\af_q$ is a subspace for $Z= G \times G / G$
as chosen in Remark~\ref{rmk-spherical}.
The element $\Lambda_V\in\af^*\times\af^*$ is identified as
$\Lambda_V=(\Lambda_W,-\Lambda_W)$, and likewise $d_V=2d_W$.
If we write $\pi$ for the action of $G$ on $W^\infty$, then it follows that
the bound in   
Theorem \ref{Sstrong upper bound} asserts for all $a\in \oline{A^+}$ and $T\in V^\infty$
that 
$$ |{\rm tr} (\pi (a^{-1}) T) |\leq  a^{\Lambda_W} ( 1 + \|\log a\|)^{2d_W} q(T)$$
with $q$ a continuous norm on $V^\infty$. 
Let us specialize to the case 
where $T$ is a rank one operator $T(w):= \tilde u(w) u$ for $u, w\in W^\infty$ 
and $\tilde u \in \tilde W^\infty$. 
We conclude:

\begin{cor}\label{group case cor}
Let $W$ be a Harish-Chandra module for $G$.
Then there exist $d\in\N$ and
continuous norms $p$ on $W^\infty$ 
and $\tilde p$ on $\tilde W^\infty$
such that 
$$ |\tilde u (\pi (a^{-1}) u)| \leq  a^{\Lambda_W} ( 1 + \|\log a\|)^{d} p(u) \tilde p(\tilde u)$$
for all $u\in W^\infty$, $\tilde u \in \tilde W^\infty$ and $ a\in \oline {A^+}$. 
\end{cor}

\begin{rmk}\label{group case rmk}
The corollary generalizes the estimate in \cite{W}, Thm.~4.3.5,
where the matrix coefficient is required to be $K$-finite
on one side. However, it should be emphasized that our proof 
of Theorem \ref{Sstrong upper bound} invokes the globalization theorem, 
which is not available at that stage in the exposition of \cite{W}.
\end{rmk}

\subsection{Other strongly spherical spaces}
We now  return to the general assumption that
$Z=G/H$ is a reductive homogeneous space with $H$ connected, and
discuss the generalization of Theorem \ref{Sstrong upper bound}.
We assume that $Z$ is strongly spherical, so that
(\ref{union j}) is valid.
Recall that $\af\subset\sf$ is maximal abelian. We use the standard
isomorphism between $\af$ and its dual space $\af^*$, and let 
$\af_{hw}\subset\af$ be the subspace such that $\af_{hw}^*$ is the span
of all the $H$-spherical highest weights $\lambda\in\af^*$.

\begin{theorem}\label{strong upper bound}
Let $V$ be a Harish-Chandra module and let $\eta\in (V^{-\infty})^H$. 
Let $P_{1},...,P_{\ell}$ be minimal parabolic subgroups that contains $A=exp(\af)$ with $P_{j}H$ open for each $1 \leq j \leq \ell$ and such that 
$Z=\bigcup_{j=1}^{\ell} K \oline{A^+_{P_j}}.z_0.$
\begin{enumerate}
\item
Then for any  $v\in V^\infty$ there exists a constant $C_{v}$ such that  
\begin{equation*} \label{fundbound3w} |m_{v,\eta}(a)| \leq C_{v} 
a^{\Lambda_{j,V}}
(1+\|\log a\|)^{d_V}\, \end{equation*} 
for all $a \in \oline{A^+_{P_j}}.$

\item 
Suppose that 

\begin{equation}\label{enough_fd_spherical}
Z=\cup_{j=1}^l K(\oline{A^+_{P_j}}\cap A_{hw}).z_0\,.
\end{equation}

Then there exists a continuous norm $q$ on $V^\infty$ 
such that 
\begin{equation*} \label{fundbound3} |m_{v,\eta}(a)| \leq q(v) 
a^{\Lambda_{j,V}}
(1+\|\log a\|)^{d_V}\, \end{equation*} 
for all $a \in \oline{A^+_{P_j}}\cap A_{hw}$, $v\in V^\infty$.

\end{enumerate}

\end{theorem}

\begin{proof}
The first point in the theorem is a direct consequence of Corollary \ref{corsst}. The second point follows from the first and the assumption (\ref{enough_fd_spherical}) using the arguments given in theorem \ref{Sstrong upper bound}. 
\end{proof}




\begin{rmk} The triple space for $\Sl(2,\R)$ satisfies 
the assumptions with $\af_{hw}=\af=\af_1\oplus\af_2\oplus\af_3$. 
If we denote by $\delta$ the defining 
representation of $\Sl(2,\R)$, then $\delta\times\delta\times\1$,
$\delta\times\1\times\delta$ and $\1\times\delta\times\delta$
are $H$-spherical representations, and their highest weights span $\af$.
\end{rmk}

\end{document}